\documentclass[10pt,reqno]{amsart}
\usepackage{amsmath,amssymb}
\usepackage{enumitem}
\usepackage{graphics,color}
\usepackage{epsfig}
\usepackage{graphicx}
\usepackage{amsfonts,amsmath,amssymb,amscd}
\setbox0=\hbox{$+$}
\newdimen\plusheight
\plusheight=\ht0
\def\+{\;\lower\plusheight\hbox{$+$}\;}

\setbox0=\hbox{$-$}
\newdimen\minusheight
\minusheight=\ht0
\def\-{\;\lower\minusheight\hbox{$-$}\;}

\setbox0=\hbox{$\cdots$}
\newdimen\cdotsheight
\cdotsheight=\plusheight
\def\cds{\lower\cdotsheight\hbox{$\cdots$}}

\renewcommand{\(}{\left\(}
\renewcommand{\)}{\right\)}

\renewcommand{\i}{\infty}

\theoremstyle{plain}
\newtheorem{theorem}{Theorem}[section]
\newtheorem{lemma}[theorem]{Lemma}

\newtheorem{corollary}[theorem]{Corollary}

\theoremstyle{definition}
\newtheorem{definition}[theorem]{Definition}

\theoremstyle{remark}
\newtheorem{remark}[theorem]{Remark}
\numberwithin{equation}{section}

\begin{document}
\title{Lauricella hypergeometric series $F_A^{(n)}$ over finite fields}
\author{Arjun Singh Chetry}
\address{Department of Science and Mathematics, Indian Institute of Information Technology Guwahati, Bongora, Assam-781015, India}
\email{achetry52@gmail.com}
\author{Gautam Kalita}
\address{Department of Science and Mathematics, Indian Institute of Information Technology Guwahati, Bongora, Assam-781015, India}
\email{gautam.kalita@iiitg.ac.in}


\begin{abstract}
In this paper, we develop a finite field analogue for one of the Lauricella series, $F^{(n)}_A$. Extending results of Greene, a finite field analogue for the multinomial coefficient is developed in order to express the Lauricella series in terms of binomial coefficients. We have further deduced certain transformation and reduction formulas for the Lauricella series $F^{(n)}_A$. Finally, we have obtained a number of generating functions for the Lauricella series $F^{(n)}_A$.
\end{abstract}
\maketitle
\section{Introduction}
In his famous paper presented to the Royal Society of Sciences at G$\ddot{o}$ttingen, Gauss \cite{gauss} introduced ${_2}F_1$-classical hypergeometric series. For complex numbers $a_i, b_j$ and $x$, with none of the $b_j$ being negative integers or zero, the classical hypergeometric series ${_{n+1}}F_n$ is defined as
\begin{align*}
_{n+1}F_n&\left[\begin{array}{cccc}
a_0, & a_1, &\ldots, & a_{n}\\
&b_1 , &\ldots, &b_n
\end{array}\vline~x \right]
:=\sum_{k=0}^{\infty}\frac{(a_0)_k\cdots(a_{n})_k}{(b_1)_k\cdots(b_n)_k}\frac{x^k}{k!},
\end{align*}
where the rising factorial $(a)_n$ is defined by $$(a)_0:=1~~~~~~\text{ and }~~~~~~(a)_k:=a(a+1)\cdots(a+k-1)~~~\text{for }k\geq 1.$$ The classical hypergeometric series satisfy many summation and transformation formulas. For details, see \cite{andrews,bailey}. The connection of the classical hypergeometric series with other number theoretical objects has been explored by many mathematician.
\par Throughout the paper, let $p$ be an odd prime and $\mathbb F_q$ denotes the finite field with $q=p^r$ elements, where $r\in\mathbb{N}$. A multiplicative character $\chi$ on $\mathbb{F}_q^\times$ is a group homomorphism $\chi:\mathbb{F}_q^\times\rightarrow\mathbb{C}$. We extend the domain of each $\chi$ on $\mathbb{F}_q^\times$ to $\mathbb{F}_q$ by setting $\chi(0) = 0$, and we denote the group of multiplicative characters on $\mathbb{F}_q$ by $\widehat{\mathbb F_q}$. For characters $A,B\in \widehat{\mathbb F_q}$, the binomial coefficient ${A \choose B}$ is defined as
\begin{align}\label{binomgreene}
{A \choose B}:=\frac{B(-1)}{q}J(A,\overline{B})=\frac{B(-1)}{q}\sum_{x \in \mathbb{F}_q}A(x)\overline{B}(1-x),
\end{align}
where $J(A, B)$ denotes the usual Jacobi sum and $\overline{B}$ is the inverse of $B$. The following properties of binomial coefficients is known from \cite{greene}
\begin{align}\label{bionomialproperty2}
\binom{A}{B}=\binom{A}{A\overline{B}},
\end{align}
\begin{align}\label{bionomialproperty1}
\binom{A}{B}=\binom{B\overline{A}}{B}B(-1),
\end{align}
\begin{align}\label{greenevar}
\overline{A}(1-x)&=\delta(x)+\frac{q}{q-1}\sum_{\chi}{A\chi\choose \chi}\chi(x),
\end{align}
where $\delta(x)=1$ $($respectively, $0)$ if $x= 0$ $($respectively, $x\neq 0)$. With these notations, for characters $A,B,C\in\widehat{\mathbb F_q}$ and $x\in\mathbb F_q$, Greene \cite{greene} defined the Gaussian hypergeometric series $_2F_1$ as
\begin{align}\label{eq02}
{_2F}_1\left[\begin{array}{cc}
A, & B\\
& C
\end{array}\mid x \right]:
=\epsilon(x)\frac{BC(-1)}{q}\sum_{y\in\mathbb{F}_q}B(y)\overline{B}C(1-y)\overline{A}(1-xy),
\end{align}
which is a finite field analogue for the integral representation of $_2F_1$-classical hypergeometric series
\begin{align*}
{_2F}_1\left[\begin{array}{cc}
a, & b\\
& c\end{array}\mid x \right]=\frac{\Gamma(c)}{\Gamma(b)\Gamma(c-b)}\int_0^1t^b(1-t)^{c-b}(1-tx)^{-a}\frac{dt}{t(1-t)}.
\end{align*}
Using \eqref{binomgreene}, Greene \cite[Theorem 3.6]{greene} expressed \eqref{eq02} as
\begin{align*}
{_{2}}F_1\left[\begin{array}{cc}
A, & B\\
& C
\end{array}\mid x \right]:
=\frac{q}{q-1}\sum_{\chi\in\widehat{\mathbb F_q}}{A\chi \choose \chi}{B\chi \choose C\chi}
\chi(x).
\end{align*}
Noting that
\begin{align*}
_{n+1}F_n&\left[\begin{array}{cccc}
a_0, & a_2, &\ldots, & a_{n}\\
&b_1 , &\ldots, &b_n
\end{array}\vline~x \right]
:=C\sum_{k=0}^{\infty}{a_0+k-1\choose k}{a_1+k-1\choose b_1+k-1}
\cdots{a_n+k-1\choose b_n+k-1}x^k,
\end{align*}
where $C=\{{a_1-1\choose b_1-1}\cdots{a_n-1\choose b_n-1}\}^{-1}$, 
Greene \cite{greene} defined the Gaussian hypergeometric series ${_{n+1}}F_n$ over $\mathbb{F}_q$ as
\begin{align*}
{_{n+1}}F_n\left[\begin{array}{cccc}
A_0, & A_1, & \ldots, & A_n\\
& B_1, & \ldots, & B_n
\end{array}\mid x \right]:
=\frac{q}{q-1}\sum_{\chi\in\widehat{\mathbb F_q}}{A_0\chi \choose \chi}{A_1\chi \choose B_1\chi}
\cdots {B_n\chi \choose B_n\chi}\chi(x),\notag
\end{align*}
where $A_0, A_1,\ldots, A_n, B_1, B_2,\ldots, B_n$ are multiplicative characters of $\mathbb{F}_q$.
\par Considering the product of two classical hypergeometric series of the form $_2F_1$, one can obtain double series. Among them, the Appell hypergeometric series $F_1,F_2,F_3$, and $F_4$ in two variables are of importance. For more details about Appell hypergeometric series, see \cite{andrews, bailey, schlosser}. Motivated by the work of Greene \cite{greene}, Li {\it et. al.} \cite{li} defined a finite field analogue for $F_1$ using its integral representation. Following this, He {\it et. al.} \cite{he} and He \cite{he2} gave finite field analogues for the Appell hypergeometric series $F_2$ and $F_3$, respectively, based on their integral representations. 
For example, He {\it et. al.} \cite{he} defined the finite field analogue for the Appell hypergeometric series
\begin{align*}
F_2&\left[\begin{array}{ccc}
a; & b, & b'\\
&c , &c'
\end{array}\vline~x,y \right]
:=\sum_{m,n\geq0}\frac{(a)_{m+n}(b)_m(b')_n}{m!n!(c)_m(c')_n}x^my^n,~|x|+|y|<1.
\end{align*}
as
\begin{align*}
F_2&\left[\begin{array}{ccc}
A; &B_1, &B_2\\
&C_1, &C_2
\end{array}\vline~x_1,x_2\right]\\
&=\epsilon(x_1x_2)B_1B_2C_1 C_2(-1)\sum_{t_1, t_2\in\mathbb{F}_q}B_1(t_1)B_2(t_2)\overline{B}_1C_1(1-t_1)\overline{B}_2C_2(1-t_2)\overline{A}(1-x_1t_1-x_2t_2).
\end{align*}
Recently, Tripathi {\it et. al.} \cite{tripathi} and Tripathi-Barman \cite{tripathi1} gave certain finite field analogues for all four Appell hypergeometric series in terms of Gauss sums, and deduced many reduction and transformation formulas for them.
\par In 1893, Lauricella \cite{lauricella} generalized all four Appell hypergeometric series, known as Lauricella series, into $n$-variables. He \cite{he1} deduced a finite field analogue for the Lauricella series $F_D^{(n)}$, and obtained certain transformation and reduction formulas together with several generating functions for them. Our main aim in this paper is to find a finite field analogue for another Lauricella series $F_A^{(n)}$, and deduce transformation and reduction formulas together with generating functions for the Lauricella series $F_A^{(n)}$ over finite fields. Lauricella \cite{lauricella} defined the Lauricella series $F_A^{(n)}$, generalization of $F_2$, as
\begin{align}\label{f2b}
F_A^{(n)}&\left[\begin{array}{cccc}
a; &b_1, &\ldots, &b_n\\
&c_1, &\ldots, &c_n
\end{array}\vline~x_1,\ldots,x_n\right]\notag\\
&=\sum_{m_1\geq0}\cdots\sum_{m_n\geq0}\frac{(a)_{m_1+\cdots+m_n}(b_1)_{m_1}\cdots(b_n)_{m_n}}{m!n!(c_1)_{m_1}\cdots(c_n)_{m_n}}x_1^{m_1}\cdots x_n^{m_n},~|x_1|+\cdots+|x_n|<1.
\end{align}
It is easy to see that $F_A^{(2)}=F_2$ and $F_A^{(1)}={_2}F_1$. 
The Lauricella series $F_A^{(n)}$ has an integral representation \cite[(24)]{hasanov}
\begin{align}\label{f2}
F_A^{(n)}&\left[\begin{array}{cccc}
a; &b_1, &\cdots, &b_n\\
&c_1, &\cdots, &c_n
\end{array}\vline~x_1,\cdots,x_n\right]=\frac{\Gamma(c_1)\cdots\Gamma(c_n)}{\Gamma(b_1)\cdots\Gamma(b_n)\Gamma(c_1-b_1)\cdots\Gamma(c_n-b_n)}\\
&\int_{0}^{1}\cdots\int_{0}^{1}t_1^{b_1-1}\cdots t_n^{b_n-1}(1-t_1)^{c_1-b_1-1}\cdots(1-t_n)^{c_n-b_n-1}(1-x_1t_1-\cdots -x_nt_n)^{-a}dt_1\cdots dt_n,\notag
\end{align}
where Re$(c_j)>$Re$(b_j)>0$ for all $j=1,\ldots,n$. Motivated by the works of \cite{greene, he, he1, he2, li}, we here first develop a finite field analogue for the Lauricella series $F_A^{(n)}$ based on the integral representation \eqref{f2}.
\begin{definition}\label{def1}
	For characters $A,B_1,\cdots,B_n,C_1,\cdots,C_n\in\widehat{\mathbb{F}_q}$ and $x_1,\cdots,x_n\in\mathbb F_q$, we define
	\begin{align*}
	F_A^{(n)}&\left[\begin{array}{cccc}
	A; &B_1, &\ldots, &B_n\\
	&C_1, &\ldots, &C_n
	\end{array}\vline~x_1,\ldots,x_n\right]\\
	&=\sum_{t_1,\ldots, t_n\in\mathbb{F}_q}\left(\prod_{i=1}^{n}\epsilon(x_i)\frac{B_iC_i(-1)}{q}B_i(t_i)\overline{B}_iC_i(1-t_i)\right)\overline{A}(1-x_1t_1-\cdots -x_nt_n).
	\end{align*}
\end{definition}
In the above definition, the normalized constant $\frac{\Gamma(c_1)\cdots\Gamma(c_n)}{\Gamma(b_1)\cdots\Gamma(b_n)\Gamma(c_1-b_1)\cdots\Gamma(c_n-b_n)}$ is dropped in order to produce simpler results, and the factor $\epsilon(x_1\cdots x_n)\frac{B_1\cdots B_nC_1\cdots C_n(-1)}{q^n}$ is introduced for a better expression of $F_A^{(n)}$ in terms of binomial coefficients.
\begin{remark}
For a permutation $\sigma:\lbrace 1,2,\cdots,n\rbrace\rightarrow\lbrace 1,2,\cdots,n\rbrace$, it follows easily from Definition \ref{def1} that
\small
\begin{align*}
F_A^{(n)}\left[\begin{array}{cccc}
A; &B_1, &\ldots, &B_n\\
&C_1, &\ldots, &C_n
\end{array}\vline~x_1,\ldots,x_n\right]=
F_A^{(n)}\left[\begin{array}{cccc}
A; &B_{\sigma(1)}, &\ldots, &B_{\sigma(n)}\\
&C_{\sigma(1)}, &\ldots, &C_{\sigma(n)}
\end{array}\vline~x_{\sigma(1)},\cdots,x_{\sigma(n)}\right].
\end{align*}
\end{remark}
\par Since the Lauricella series $F_A^{(n)}$ over finite fields is an $n$-variable extension of the finite field Appell series $F_2$, most of the results in this paper are generalizations of the results of \cite{he, ma}. However, we use Definition \ref{def1} to deduce these results whereas He {\it et. al.} \cite{he} used the binomial coefficient expression of $F_2$ to deduce them. It is also to note that 
our results shall be simpler compared to those proved by He {\it et. al.} \cite{he}.
\par The organization of this paper is as follows. In Section 2, we prove an expression for the Lauricella series $F_A^{(n)}$ over finite fields in terms of binomial coefficients. For this, we develop the concept of multinomial coefficients over finite fields, extending the work of Greene \cite{greene}.  We deduce certain reduction and transformation formulas for the Lauricella series $F_A^{(n)}$ over finite fields in Section 3. In Section 4, we find several generation functions for the Lauricella series $F_A^{(n)}$ over finite fields.
\section{Another Expression}
In this section, we give an expression of the finite field Lauricella series $F_A^{(n)}$ in terms of binomial coefficients.
\begin{theorem}\label{exbinom}
Let $A,B_1,\cdots,B_n,C_1,\cdots,C_n\in\widehat{\mathbb{F}_q}$ and $x_1,\cdots,x_n\in\mathbb F_q$, then
{\tiny
\begin{align*}
F_A^{(n)}&\left[\begin{array}{cccc}
A; &B_1, &\ldots, &B_n\\
   &C_1, &\ldots, &C_n
\end{array}\vline~x_1,\ldots,x_n\right]\\
&=\frac{q^n}{(q-1)^n}\sum_{\chi_1,\ldots,\chi_n\in\widehat{\mathbb F_q}}{A\chi_1\choose\chi_1}{A\chi_1\chi_2\choose\chi_2}\cdots{A\chi_1\chi_2\cdots\chi_n\choose\chi_n}{B_1\chi_1\choose C_1\chi_1}{B_2\chi_2\choose C_1\chi_2}\cdots{B_n\chi_n\choose C_n\chi_n}\chi_1(x_1)\cdots\chi_n(x_n).
\end{align*}}
\end{theorem}
For $n=2$, Theorem \ref{exbinom} gives a simple expression for the Appell series $F_A^{(2)}$ over finite fields. 
\begin{corollary}
Let $A,B,B',C,C'\in\widehat{\mathbb{F}_q}$ and $x,y\in\mathbb F_q$, then
\begin{align*}
F_A^{(2)}\left[\begin{array}{ccc}
A; &B, &B'\\
&C, &C'
\end{array}\vline~x,y\right]=\frac{q^2}{(q-1)^2}\sum_{\chi,\lambda\in\widehat{\mathbb F_q}}{A\chi\choose\chi}{A\chi\lambda\choose\lambda}{B\chi\choose C\chi}{B'\lambda\choose C'\lambda}\chi(x)\lambda(y).
\end{align*}
\end{corollary}
To prove Theorem \ref{binomgreene}, we follow the work of Greene \cite{greene} to establish certain results concerning multinomial coefficients over finite fields.
\par Let $\chi_1,\ldots,\chi_n\in\widehat{\mathbb F_q}$. Any function $f:\mathbb F_q^n\rightarrow\mathbb C$ has a unique representation
{\tiny
\begin{align}\label{exbinomeq1}
f(x_1,\ldots,x_n)&=f_\delta\delta(x_1)\cdots\delta(x_n)+ 
\sum_{k=1}^{n}\left(\sum_{\substack{r_1,\ldots,r_k=1\\1\leq r_1<\cdots<r_k\leq n}}
^{n}\delta_{r_1\cdots r_k}
\sum_{\chi_1,\cdots,\chi_k\in\widehat{\mathbb F_q}}f_{\chi_{r_1}\cdots\chi_{r_k}}\chi_1(x_{r_1})\cdots\chi_k(x_{r_k})\right),
\end{align}}
where
$f_\delta=f(0,\ldots,0);~~\delta(x)=\left\{
\begin{array}{ll}
\displaystyle 1 & \hbox{if $x=0$;}\\
0 & \hbox{otherwise;}
\end{array}
\right.
\displaystyle\delta_{r_1,\ldots,r_k}=\prod_{\substack{i=1\\
		i\neq r_1,\cdots,r_k}}^{n}\delta(x_i),
	$ for $k=1,2,\ldots, (n-1)$; $\delta_{r_1,\ldots,r_n}=1$;
and $
\displaystyle f_{\chi_{r_1}\cdots\chi_{r_k}}=\frac{1}{(q-1)^k}\sum_{t_1,\cdots, t_k\in\mathbb F_q}f(0,\ldots,0,t_{r_1},0,\ldots,0,t_{r_k},0,\ldots,0)\overline{\chi}_1(t_{r_1})\cdots\overline{\chi}_k(t_{r_k}),\\
$
with $t_{r_i}$ is at the $r_i$th position in the tuple.
\par In \cite[Definition  5.18]{lidl}, for $\lambda_1,\ldots,\lambda_k\in\widehat{\mathbb F_q}$, the {\it multiple}-Jacobi sum is defined as
\begin{align*}
J(\lambda_1,\ldots,\lambda_k)&=\sum_{\substack{c_1,\ldots,c_k\in\mathbb{F}_q\\c_1+\cdots+c_k=1}}\lambda_1(c_1)\cdots\lambda_k(c_k)\\
&=\sum_{\substack{c_2,\ldots,c_k\in\mathbb{F}_q}}\lambda_1(1+c_2+\cdots+c_k)\lambda_2(-c_2)\cdots\lambda_k(-c_k).
\end{align*}
\subsection{Multinomial coefficients and multinomial theorem}
The multinomial theorem, a generalization of the binomial theorem, states that
\begin{equation}\label{multieq1}
(1+x_1+\cdots+x_n)^m=\sum_{k_1,\ldots,k_n\geq0}{m\choose k_1,\cdots,k_{n+1}}x_1^{k_2}\cdots x_n^{k_{n+1}},
\end{equation}
where $k_1+\cdots+k_{n+1}=m$ and ${m\choose k_1,\ldots,k_{n+1}}$ is defined as
$${m\choose k_1,\cdots,k_{n+1}}=\frac{m!}{k_1!\cdots k_n!k_{n+1}!}.$$
We now establish a finite field analogue of \eqref{multieq1}.
\begin{lemma}\label{multilemm1}
	For any character $A$ of $\mathbb F_q$ and $x_1,\ldots,x_n\in\mathbb F_q$, we have
	{\tiny
	\begin{align*}
	 A(1+x_1+\cdots+x_n)&=\delta(x_1)\cdots\delta(x_n)
	 +\sum_{k=1}^{n}\left(\frac{1}{(q-1)^k}\sum_{\substack{r_1,\ldots,r_k=1\\1\leq r_1<\cdots<r_k\leq n}}
	^{n}\delta_{r_1\cdots r_k}
	\sum_{\chi_1\cdots\chi_r}J(A,\overline{\chi}_1,\ldots,\overline{\chi}_k)\chi_1(-x_{r_1})\cdots\chi_k(-x_{r_k})\right).
	\end{align*}}
\end{lemma}
\begin{proof}
Let $f(x_1,\cdots,x_n)=A(1+x_1+\cdots+x_n)$. Then we have
	\begin{align*}
	f_\delta&=A(1)=1,
	\end{align*}
	and
	\begin{align*}
	f_{\chi_{r_1}\cdots\chi_{r_k}}&=\frac{1}{(q-1)^k}\sum_{t_1,\ldots, t_k\in\mathbb F_q}A(1+t_1+\cdots+t_k)\overline{\chi}_1(t_1)\cdots\overline{\chi}_k(t_k)\\
	&=\frac{\chi_1\cdots\chi_k(-1)}{(q-1)^n}J(A,\overline{\chi}_1,\ldots,\overline{\chi}_k)
	\end{align*}
	for all $k=1,2,\ldots, n$.
Thus the result follows directly from \eqref{exbinomeq1}.
\end{proof}
It is easy to see from \eqref{multieq1} and Lemma \ref{multilemm1} that the finite field analogue for the multinomial coefficients is the {\it multiple}-Jacobi sum. This leads to the following definition.
\begin{definition}\label{def2}
	Let $A,B_1,\ldots,B_n\in\widehat{\mathbb F_q}$, define ${A\choose B_1\cdots B_n}$ as
	$${A\choose B_1,\ldots,B_n}:=\frac{B_1\cdots B_n(-1)}{q^n}J(A,\overline{B}_1,\cdots,\overline{B}_n).$$
\end{definition}
Thus in terms of multinomial coefficients, Theorem \ref{multilemm1} can be rewritten as
{\tiny
\begin{align}\label{defeq1}
A(1+x_1+\cdots+x_n)&=\delta(x_1)\cdots\delta(x_n)
+\sum_{k=1}^{n}\left(\frac{q^k}{(q-1)^k}\sum_{\substack{r_1,\ldots,r_k=1\\1\leq r_1<\cdots<r_k\leq n}}
^{n}\delta_{r_1\cdots r_k}
\sum_{\chi_1,\ldots,\chi_r\in\widehat{\mathbb F_q}}{A\choose \chi_1,\cdots,\chi_k}\chi_1(x_{r_1})\cdots\chi_k(x_{r_k})\right).
\end{align}}
The multinomial co-efficient is related to the binomial co-efficient as
\begin{align}\label{defeq2}
{m\choose k_1,\ldots,k_n}&=\frac{m!}{k_1!\cdots k_n!(m-k_1-\cdots-k_{n})!}={m\choose k_1}{m-k_1\choose k_1}\cdots{m-k_1-\cdots-k_{n-1}\choose k_n}.
\end{align}
The following result is a finite field analogue of \eqref{defeq2}.
\begin{lemma}\label{defeq3}
For $A,B_1,\ldots,B_n\in\widehat{\mathbb F_q}$, we have
\begin{align}
{A\choose B_1,\ldots,B_n}&={A\choose B_1}{A\overline{B}_1\choose B_2}{A\overline{B_1B_2}\choose B_3}\cdots{A\overline{B_1\cdots B_{n-1}}\choose B_n}\notag \\&
=B_1\cdots B_n(-1){\overline{A}B_1\choose B_1}{\overline{A}B_1B_2\choose B_2}
\cdots{\overline{A}B_1\cdots B_n\choose B_n}.\notag
\end{align}
\end{lemma}
\begin{proof}
Using \eqref{bionomialproperty1} and \eqref{binomgreene}, we have
\begin{align*}
{\overline{A}B_1\choose B_1}{\overline{A}B_1B_2\choose B_2}
\cdots{\overline{A}B_1\cdots B_n\choose B_n}&=B_1\cdots B_n(-1){A\choose B_1}
{A\overline{B}_1\choose B_2}{A\overline{B_1B_2}\choose B_3}\cdots{A\overline{B_1\cdots B_{n-1}}\choose B_n}\\
&=\frac{1}{q^n}
\sum_{t_1\in\mathbb{F}_q}A(t_1)\overline{B_1}(1-t_1)\sum_{t_2\in\mathbb{F}_q}A\overline{B_1}(t_2)\overline{B_2}(1-t_2)\\
&\hspace{2cm}\cdots\sum_{t_n}A\overline{B_1\cdots B_{n-1}}(t_n)\overline{B_n}(1-t_n)\\
&=\frac{1}{q^n}
\sum_{t_1,\ldots,t_n\in\mathbb{F}_q}A(t_1\cdots t_n)\overline{B_1}((1-t_1)t_2\cdots t_n)\\
&\hspace{1.6cm}\overline{B_2}((1-t_2)t_3\cdots t_n)\cdots\overline{B_{n-1}}((1-t_{n-1})t_n)\overline{B_n}(1-t_n).
\end{align*}
Noting that $t_1\cdots t_n+(1-t_1)t_2\cdots t_n+\cdots+(1-t_{n-1})t_n+(1-t_n)=1$, the result follows from Definition \ref{def2}.
\end{proof}
\noindent\textbf{Proof of Theorem \ref{exbinom}}. Using Lemma \eqref{defeq3} in \eqref{defeq1}, 
we have
\begin{align*}
&\overline{A}(1-x_1t_1-\cdots-x_nt_n)
=\delta(x_1t_1)\cdots\delta(x_nt_n)+\\
&\sum_{k=1}^{n}\left(\frac{q^k}{(q-1)^k}\sum_{\substack{r_1,\ldots,r_k=1\\1\leq r_1<\cdots<r_k\leq n}}
^{n}\delta_{r_1\cdots r_k}
\sum_{\chi_1,\ldots,\chi_r\in\widehat{\mathbb F_q}}{A\chi_1\choose \chi_1}{A\chi_1\chi_2\choose \chi_2}\cdots{A\chi_1\cdots\chi_k\choose \chi_k}\chi_1(x_{r_1}t_{r_1})\cdots\chi_k(x_{r_k}t_{r_k})\right),
\end{align*}
where $$\delta_{r_1,\ldots,r_k}=\prod_{\substack{i=1\\i\neq r_1,\ldots,r_k}}^{n}\delta(x_it_i).$$ Since $$\epsilon(x_1,\ldots, x_n)\delta_{r_1,\ldots,r_k}\prod_{i=1}^{n}B_i(t_i)=0,$$ Definition \ref{def1} yields
\begin{align*}
F_A^{(n)}&\left[\begin{array}{cccc}
A; &B_1, &\ldots, &B_n\\
&C_1, &\ldots, &C_n
\end{array}\vline~x_1,\ldots,x_n\right]\\
&=\frac{q^n}{(q-1)^n}\sum_{t_1,\ldots, t_n\in\mathbb{F}_q}\left(\prod_{i=1}^{n}\epsilon(x_i)\frac{B_iC_i(-1)}{q}B_i(t_i) \overline{B}_iC_i(1-t_i)\right)\\
&\hspace{3cm}\sum_{\chi_1,\ldots,\chi_n\in\widehat{\mathbb F_q}}{A\chi_1\choose \chi_1}{A\chi_1\chi_2\choose \chi_2}\cdots{A\chi_1\cdots\chi_n\choose \chi_n}\chi_1(x_{1}t_{1})\cdots\chi_n(x_{n}t_{n})\\
&=\frac{1}{(q-1)^n}\sum_{\chi_1,\ldots,\chi_n\in\widehat{\mathbb F_q}}\sum_{t_1,\ldots, t_n\in\mathbb{F}_q}\left(\prod_{i=1}^{n}B_iC_i(-1)B_i\chi_i(t_i)
\overline{B}_iC_i(1-t_i)
\right)\\
&\hspace{4cm}{A\chi_1\choose \chi_1}{A\chi_1\chi_2\choose \chi_2}\cdots{A\chi_1\cdots\chi_k\choose \chi_k}\chi_1(x_1)\cdots\chi_n(x_n)\\
&=\frac{q^n}{(q-1)^n}\sum_{\chi_1\cdots\chi_n}{A\chi_1\choose \chi_1}{A\chi_1\chi_2\choose \chi_2}\cdots{A\chi_1\cdots\chi_k\choose \chi_k}{B_1\chi_1\choose \overline{C}_1B_1}\cdots {B_n\chi_n\choose \overline{C}_nB_n}\chi_1(x_1)\cdots\chi_n(x_n),
\end{align*}
where the last equality follows due to \eqref{binomgreene}. Using \eqref{bionomialproperty2}, we complete the proof of the result.\hfill$\Box$
\section{Reduction and Transformation formulas}
In this section, we deduce certain reduction and transformation formulas for the Lauricella series $F_A^{(n)}$. We first extend some results from \cite{brychkov} of Appell series $F_2$ to $F_A^{(n)}$ in the following lemma.
	\begin{lemma}\label{maid}
	For $a,b_1,\ldots,b_n,c_1,\ldots,c_n,x_1,\ldots,x_n,t\in\mathbb C$, the following identities hold.
	{\tiny
		\begin{align*}
	&(i)~F_A^{(n)}\left[\begin{array}{cccc}
	a; &b_1, &\ldots, &b_n\\
	&c_1, &\cdots, &c_n
	\end{array}\vline~x_1,\ldots,x_n\right]\\
	 &\hspace{.5cm}=\frac{\Gamma(c_k)}{\Gamma(b_k)\Gamma(c_k-b_k)}\int_{0}^{1}t_k^{b_k-1}(1-t_k)^{c_k-b_k-1}(1-x_kt_k)^{-a}\\
	&\hspace{1.5cm} F_A^{(n-1)}\left[\begin{array}{ccccccc}
	a; &b_1, &\ldots, &b_{k-1}, &b_{k+1}, &\cdots, &b_n\\
	&c_1, &\ldots, &c_{k-1}, &c_{k+1}, &\ldots, &c_n
	\end{array}\vline~\frac{x_1}{1-x_kt_k},\ldots,\frac{x_{k-1}}{1-x_kt_k},\frac{x_{k+1}}{1-x_kt_k} \cdots,\frac{x_n}{1-x_kt_k}\right]dt_k.\\
	&(ii)~F_A^{(n)}\left[\begin{array}{cccc}
	a; &b_1, &\ldots, &b_n\\
	&c_1, &\ldots, &c_n
	\end{array}\vline~x_1,\ldots,x_n\right]\\
	 &\hspace{.5cm}=\frac{\Gamma(c_k)}{\Gamma(b_k)\Gamma(c_k-b_k)}\left(\frac{1}{x_k}\right)^{c_k-1}\int_{1}^{\frac{1}{1-x_k}}t_k^{a-c_k}(t_k-1)^{b_k-1}(1-t_k+x_kt_k)^{c_k-b_k-1}\\
	&\hspace{1.5cm} F_A^{(n-1)}\left[\begin{array}{ccccccc}
	a; &b_1, &\ldots, &b_{k-1}, &b_{k+1}, &\cdots, &b_n\\
	&c_1, &\ldots, &c_{k-1}, &c_{k+1}, &\cdots, &c_n
	\end{array}\vline~x_1t_k,\ldots,x_{k-1}t_k,x_{k+1}t_k \cdots,x_nt_k\right]dt_k.\\
	&(iii)~F_A^{(n)}\left[\begin{array}{cccc}
	a; &b_1, &\ldots, &b_n\\
	&c_1, &\ldots, &c_n
	\end{array}\vline~x_1t,\ldots x_{k-1}t,1-\frac{t}{x_k},x_{k+1}t,\ldots,x_nt\right]\\
	&\hspace{.5cm}=\frac{\Gamma(c_k)}{\Gamma(b_k)\Gamma(c_k-b_k)}t^{c_k-b_k-a}x_k^{b_k}(x_k-t)^{1-c_k}\int_{t}^{x_k}t_k^{a-c_k}(t_k-t)^{b_k-1}(x_k-t_k)^{c_k-b_k-1}\\
&\hspace{1.5cm} F_A^{(n-1)}\left[\begin{array}{ccccccc}
	a; &b_1, &\ldots, &b_{k-1}, &b_{k+1}, &\ldots, &b_n\\
	&c_1, &\ldots, &c_{k-1}, &c_{k+1}, &\ldots, &c_n
	\end{array}\vline~x_1t_k,\ldots,x_{k-1}t_k,x_{k+1}t_k \ldots,x_nt_k\right]dt_k.
	\end{align*}}
\end{lemma}
\begin{proof}
	$(i)$ Changing the order of integration slightly in \eqref{f2}, we obtain
		\begin{align*}
F_A^{(n)}&\left[\begin{array}{cccc}
	a; &b_1, &\ldots, &b_n\\
	&c_1, &\ldots, &c_n
	\end{array}\vline~x_1,\ldots,x_n\right]\\ &=\frac{\Gamma(c_1)\cdots\Gamma(c_n)}{\Gamma(b_1)\cdots\Gamma(b_n)\Gamma(c_1-b_1)\cdots\Gamma(c_n-b_n)}\int_{0}^{1}t_k^{b_k-1}(1-t_k)^{c_k-b_k-1}(1-x_kt_k)^{-a}\\
	&\hspace{.2cm}\int_0^1\cdots\int_{0}^{1}t_1^{b_1-1}\cdots t_n^{b_n-1}(1-t_1)^{c_1-b_1-1}\cdots(1-t_n)^{c_n-b_n-1}\left(1-\frac{x_1t_1}{1-x_kt_k}-\cdots -\frac{x_nt_n}{1-x_kt_k}\right)^{-a}dt_1\cdots dt_n dt_k.
		\end{align*}
\noindent Hence the result follows.\\
	$(ii)$ Changing the variables $t_k\rightarrow\frac{t_k-1}{x_kt_k}$ in $(i)$, we obtain the desired result.\\
	$(iii)$ We first change the variables $x_j\rightarrow x_jt$ for all $j\neq k$, $x_k\rightarrow1-\frac{t}{x_k}$, followed by $t_k\rightarrow\frac{t_k}{t}$ in $(ii)$, the result follows.
\end{proof}
The finite field analogue of the above identities are as follows.
\begin{theorem}\label{thmma2.1}
	For $n\geq2$, let $A,B_1,\ldots,B_n,C_1,\ldots,C_n\in\widehat{\mathbb{F}_q}$ and $x_1,\ldots,x_n\in\mathbb F_q$. If $k\neq l$ such that $x_l$ is nonzero, then
	\small
	\begin{align*}
	F_A^{(n)}&\left[\begin{array}{cccc}
	A; &B_1, &\ldots, &B_n\\
	&C_1, &\ldots, &C_n
	\end{array}\vline~x_1,\ldots,x_n\right]\\
	&=\frac{B_kC_k(-1)}{q}\overline{B}_kC_k(x_k-1)\overline{C}_k(x_k)\overline{A}(-x_l)C_l(-1)\\
	&\hspace{.3cm} F_A^{(n-1)}\left[\begin{array}{cccccc}
	A; &B_1, &\ldots &\overline{B}_{l}C_l, &\ldots , &B_n\\
	&C_1, &\ldots, &C_{l}, &\ldots, &C_n
	\end{array}\vline~-\frac{x_1}{x_l},\ldots,-\frac{x_{l-1}}{x_l},1,-\frac{x_{l+1}}{x_l},\ldots,-\frac{x_n}{x_l}\right]\\
	&\hspace{.3cm}+\frac{B_kC_k(-1)}{q}\sum_{t_k\neq x_k^{-1}}B_k(t_k)\overline{B}_kC_k(1-t_k)\overline{A}(1-x_kt_k)\\
	&\hspace{.3cm} F_A^{(n-1)}\left[\begin{array}{ccccccc}
	A; &B_1, &\ldots &B_{l-1}, &B_{l+1}, &\ldots , &B_n\\
	&C_1, &\ldots, &C_{l-1}, &C_{l+1}, &\ldots, &C_n
	\end{array}\vline~\frac{x_1}{1-x_kt_k},\ldots,\frac{x_{k-1}}{1-x_kt_k},\frac{x_{k+1}}{1-x_kt_k},\ldots,\frac{x_n}{1-x_kt_k}\right].
	\end{align*}
\end{theorem}
\begin{proof} For some fixed $k$ with $1\leq k\leq n$, we divide the series into two sums	
	\begin{align}\label{3.1}
	F_A^{(n)}\left[\begin{array}{cccc}
	A; &B_1, &\ldots, &B_n\\
	&C_1, &\ldots, &C_n
	\end{array}\vline~x_1,\ldots,x_n\right]=&\epsilon(x_1\cdots x_n)\frac{B_1\cdots B_nC_1\cdots C_n(-1)}{q^n}\left(\sum_{\substack{x_kt_k=1\\t_i\in\mathbb F_q}}+\sum_{\substack{x_kt_k\neq1\\t_i\in\mathbb F_q}}\right)\notag\\
	:=&P+Q.
	\end{align}
We now evaluate each expression separately. We choose $l$ with $1\leq l\leq n$ such that $l\neq k$ and $x_l\neq 0$. Substituting $t_l=1-t_l'$ in the expression of $P$, we obtain
{\tiny	\begin{align*}
	P&=\sum_{t_1,\ldots, t_n\in\mathbb{F}_q}\left(\prod_{\substack{i=1\\i\neq k}}^{n}\epsilon(x_i)\frac{B_iC_i(-1)}{q}B_i(t_i)\overline{B}_iC_i(1-t_i)\right)\frac{B_kC_k(-1)}{q}B_k(x_k^{-1})\overline{B}_kC_k(1-x_k^{-1})\overline{A}(-x_1t_1-\cdots-x_lt_l-\cdots-x_nt_n)\\
	&=\sum_{t_1,\ldots, t_n\in\mathbb{F}_q}\left(\prod_{\substack{i=1\\i\neq k,l}}^{n}\epsilon(x_i)\frac{B_iC_i(-1)}{q}B_i(t_i)\overline{B}_iC_i(1-t_i)\right)\frac{B_kC_kB_lC_l(-1)}{q}\overline{B}_kC_k(x_k-1)\overline{C}_k(x_k)B_l(1-t_l')\overline{B}_lC_l(t_l')\\
	&\hspace{7cm}\overline{A}\left(-x_1t_1-\cdots-x_l+x_lt_l'-\cdots-x_nt_n\right)\\
	&=\sum_{t_1,\ldots, t_n\in\mathbb{F}_q}\left(\prod_{\substack{i=1\\i\neq k,l}}^{n}\epsilon(x_i)\frac{B_iC_i(-1)}{q}B_i(t_i)\overline{B}_iC_i(1-t_i)\right)\frac{B_kC_kB_lC_l(-1)}{q}\overline{B}_kC_k(x_k-1)\overline{C}_k(x_k)B_l(1-t_l')\overline{B}_lC_l(t_l')\overline{A}(-x_l)\\
	&\hspace{5cm}\overline{A}\left(1-t_l'+ \frac{x_1t_1}{x_l}-\cdots+\frac{x_{l-1}t_{l-1}}{x_l}+\frac{x_{l+1}t_{l+1}}{x_l}+\cdots+\frac{x_nt_n}{x_l}\right)\\
	&=\frac{B_kC_k(-1)}{q}\overline{B}_kC_k(x_k-1)\overline{C}_k(x_k)\overline{A}(-x_l)C_l(-1)F_A^{(n-1)}\left[\begin{array}{cccccc}
	A; &B_1, &\ldots, &\overline{B}_{l}C_l, &\ldots , &B_n\\
	&C_1, &\ldots, &C_{l}, &\ldots, &C_n
	\end{array}\vline~-\frac{x_1}{x_l},\ldots,-\frac{x_{l-1}}{x_l},1,-\frac{x_{l+1}}{x_l},\ldots,-\frac{x_n}{x_l}\right].
	\end{align*}
{\large In the similar way,}
	\begin{align*}
	Q=&\sum_{t_1,\cdots, t_n\in\mathbb{F}_q}\left(\prod_{i=1}^{n}\epsilon(x_i)\frac{B_iC_i(-1)}{q}B_i(t_i)\overline{B}_iC_i(1-t_i)\right)\overline{A}(1-x_kt_k)\overline{A}\left(1-\frac{x_1t_1}{1-x_kt_k}-\cdots-\frac{x_{k-1}t_{k-1}}{1-x_kt_k}-
	\frac{x_{k+1}t_{k+1}}{1-x_kt_k}\cdots-\frac{x_nt_n}{1-x_kt_k}\right)\\
	=&\frac{B_kC_k(-1)}{q}\sum_{t_k\neq x_k^{-1}}B_k(t_k)\overline{B}_kC_k(1-t_k)\overline{A}(1-x_kt_k)\sum_{\substack{t_1,\cdots,t_{k-1},t_{k+1},\ldots, t_n\in\mathbb{F}_q}}\left(\prod_{\substack{i=1\\i\neq k}}^{n}\epsilon(x_i)\frac{B_iC_i(-1)}{q}B_i(t_i)\overline{B}_iC_i(1-t_i)\right)\\
	&\hspace{5cm}\overline{A}\left(1-\frac{x_1t_1}{1-x_kt_k}-\cdots-\frac{x_{k-1}t_{k-1}}{1-x_kt_k}-
	\frac{x_{k+1}t_{k+1}}{1-x_kt_k}\cdots-\frac{x_nt_n}{1-x_kt_k}\right)\\
	=&\frac{B_kC_k(-1)}{q}\sum_{t_k\neq x_k^{-1}}B_k(t_k)\overline{B}_kC_k(1-t_k)\overline{A}(1-x_kt_k)\\
	&\hspace{1cm} F_A^{(n-1)}\left[\begin{array}{ccccccc}
	A; &B_1, &\ldots &B_{l-1}, &B_{l+1}, &\ldots , &B_n\\
	&C_1, &\ldots, &C_{l-1}, &C_{l+1}, &\ldots, &C_n
	\end{array}\vline~\frac{x_1}{1-x_kt_k},\ldots,\frac{x_{k-1}}{1-x_kt_k},\frac{x_{k+1}}{1-x_kt_k}\ldots\frac{x_n}{1-x_kt_k}\right].
	\end{align*}
{\large As a result, we obtain the desired result from \eqref{3.1}.}}
\end{proof}
If we use the change of variables of the proofs of Lemma \ref{maid} $(ii)$ and $(iii)$ in Theorem \ref{thmma2.1}, we have the following results.
\begin{theorem}\label{thmma2.2}
	For $n\geq2$, let $A,B_1,\ldots,B_n,C_1,\ldots,C_n\in\widehat{\mathbb{F}_q}$ and $x_1,\ldots,x_n\in\mathbb F_q$. If $k\neq l$ and $x_l\neq 0$, then
	{\tiny
	\begin{align*}
	F_A^{(n)}&\left[\begin{array}{cccc}
	A; &B_1, &\ldots, &B_n\\
	&C_1, &\ldots, &C_n
	\end{array}\vline~x_1,\ldots,x_n\right]\\
	&=\frac{B_kC_k(-1)}{q}\overline{B}_kC_k(x_k-1)\overline{C}_k(x_k)\overline{A}(-x_l)C_l(-1)\\
	&\hspace{1cm}F_A^{(n-1)}\left[\begin{array}{cccccc}
	A; &B_1, &\ldots, &\overline{B}_{l}C_l, &\ldots , &B_n\\
	&C_1, &\ldots, &C_{l}, &\ldots, &C_n
	\end{array}\vline~-\frac{x_1}{x_l},\ldots,-\frac{x_{l-1}}{x_l},1,-\frac{x_{l+1}}{x_l},\ldots,-\frac{x_n}{x_l}\right]\\
	&\hspace{1.2cm}+\frac{B_kC_k(-1)\overline{C}_k(x_k)}{q}\sum_{t_k}A\overline{C}_k(t_k)B_k(t_k-1)\overline{B}_kC_k(1-t_k+x_kt_k)\\
	&\hspace{1.4cm}F_A^{(n-1)}\left[\begin{array}{ccccccc}
	A; &B_1, &\ldots &B_{l-1}, &B_{l+1}, &\ldots , &B_n\\
	&C_1, &\ldots, &C_{l-1}, &C_{l+1}, &\ldots, &C_n
	\end{array}\vline~x_1t_k,\ldots,x_{k-1}t_k,x_{k+1}t_k,\ldots,x_nt_k\right].
	\end{align*}}
\end{theorem}
\begin{theorem}\label{thmma2.3}
	For $n\geq2$, let $A,B_1,\ldots,B_n,C_1,\ldots,C_n\in\widehat{\mathbb{F}_q}$ and $x_1,\ldots,x_n,t\in\mathbb F_q$. If $k\neq l$ and $x_l\neq 0$, then
	{\tiny
	\begin{align*}
	F_A^{(n)}&\left[\begin{array}{cccc}
	A; &B_1, &\ldots, &B_n\\
	&C_1, &\ldots, &C_n
	\end{array}\vline~x_1t,\ldots x_{k-1}t,1-\frac{t}{x_k},x_{k+1}t,\ldots,x_nt\right]\\
	&=\frac{\overline{A}B_kC_kC_l(-1)}{q}\overline{A}(x_lt)B_k(x_k)\overline{B}_k(x_k-t)\\
	&\hspace{1cm} F_A^{(n-1)}\left[\begin{array}{cccccc}
	A; &B_1, &\cdots &\overline{B}_{l}C_l, &\cdots , &B_n\\
	&C_1, &\cdots, &C_{l}, &\cdots, &C_n
	\end{array}\vline~-\frac{x_1t}{x_l},\cdots,-\frac{x_{l-1}t}{x_l},1,-\frac{x_{l+1}t}{x_l},\cdots,-\frac{x_nt}{x_l}\right]\\
	&\hspace{1.2cm}+\frac{B_kC_k(-1)}{q}\overline{C}_k(x_k-t)B_k(x_k)\overline{AB}_kC_k(t)\sum_{t_k}A\overline{C}_k(t_k)B_k(t_k-t)\overline{B}_kC_k(x_k-t_k)\\
	& \hspace{1.4cm}F_A^{(n-1)}\left[\begin{array}{ccccccc}
	A; &B_1, &\cdots &B_{l-1}, &B_{l+1}, &\cdots , &B_n\\
	&C_1, &\cdots, &C_{l-1}, &C_{l+1}, &\cdots, &C_n
	\end{array}\vline~x_1t_k,\cdots,x_{k-1}t_k,x_{k+1}t_k,\cdots,x_nt_k\right].
	\end{align*}}
\end{theorem}
\begin{lemma}\label{ext}
	For $x_k\neq1$, we have
	{\tiny
	\begin{align*}
	&(i)~F_A^{(n)}\left[\begin{array}{cccccccc}
	a; &b_1, &\ldots, &b_{k-1}, &0, &b_{k+1} &\ldots &b_n\\
	&c_1, &\ldots, &c_{k-1}, &c_k, &c_{k+1} &\ldots, &c_n
	\end{array}\vline~x_1,\ldots,x_n\right]\\
	&\hspace{2cm}=F_A^{(n-1)}\left[\begin{array}{ccccccc}
	a; &b_1, &\ldots, &b_{k-1}, &b_{k+1} &\cdots &b_n\\
	&c_1, &\ldots, &c_{k-1}, &c_{k+1} &\cdots, &c_n
	\end{array}\vline~x_1,\ldots,x_{k-1},x_{k+1},\ldots,x_n\right].\\
	&(ii)~F_A^{(n)}\left[\begin{array}{cccc}
	a; &b_1, &\ldots, &b_n\\
	&c_1, &\ldots, &c_n
	\end{array}\vline~x_1,\cdots,x_n\right]=(1-x_k)^{-a}\\
	&\hspace{.5cm} F_A^{(n)}\left[\begin{array}{cccccccc}
	a; &b_1, &\ldots, &b_{k-1}, &c_k-b_k, &b_{k+1}, &\ldots, &b_n\\
	&c_1, &\ldots, &c_{k-1}, &c_k, &c_{k+1}, &\ldots, &c_n
	\end{array}\vline~\frac{x_1}{1-x_k},\ldots,\frac{x_{k-1}}{1-x_k},-\frac{x_k}{1-x_k},\frac{x_{k+1}}{1-x_k}\ldots,\frac{x_n}{1-x_k}\right].\\
	&(iii)~F_A^{(n)}\left[\begin{array}{cccccccc}
	a; &b_1, &\ldots, &b_{k-1}, &b_k, &b_{k+1}, &\ldots, &b_n\\
	&c_1, &\ldots, &c_{k-1}, &b_k, &c_{k+1}, &\ldots, &c_n
	\end{array}\vline~x_1,\ldots,x_n\right]\\
	&\hspace{1.5cm} =(1-x_k)^{-a}F_A^{(n-1)}\left[\begin{array}{ccccccc}
	a; &b_1, &\ldots, &b_{k-1}, &b_{k+1}, &\ldots, &b_n\\
	&c_1, &\ldots, &c_{k-1}, &c_{k+1}, &\ldots, &c_n
	\end{array}\vline~\frac{x_1}{1-x_k},\ldots,\frac{x_{k-1}}{1-x_k},\frac{x_{k+1}}{1-x_k},\ldots,\frac{x_n}{1-x_k}\right].
	\end{align*}}
\end{lemma}
\begin{proof} The proof of $(i)$ easily follows from \eqref{f2b}. Using the substitution $t_k=1-t_k'$ for any $k$ with $1\leq k\leq n$, and $t_j=t_j'$ for all $j$ with $1\leq j\leq n,j\neq k$, the reduction formula $(ii)$ can be easily obtained from \eqref{f2}. Considering $c_k=b_k$ in $(ii)$, and then using $(i)$, we complete the proof of $(iii)$.
\end{proof}
Motivated by the work  Li {\it et. al.} \cite{li}, we now establish the finite field analogue of the identities of Lemma \ref{ext}. The proof of the results relies on Definition \ref{def1} of finite field Lauricella series $F^{(n)}_A$.
\begin{theorem}\label{eps}
For $n\geq2$, let $A,B_1,\ldots,B_n,C_1,\ldots,C_n\in\widehat{\mathbb{F}_q}$ and $x_1,\ldots,x_n\in\mathbb F_q$. If $x_k\neq1$, then
{\tiny
\begin{align*}
&F_A^{(n)}\left[\begin{array}{cccccccc}
A; &B_1, &\ldots,&B_{k-1},&\epsilon, &B_{k+1},&\ldots, &B_n\\
&C_1, &\ldots,&C_{k-1},&C_k, &C_{k+1},&\ldots, &C_n
\end{array}\vline~x_1,\ldots,x_n\right]\\
&\hspace{1cm}=\frac{\epsilon(x_k)}{q}C_k(-1) F_A^{(n-1)}\left[\begin{array}{ccccccc}
A\overline{C_k}; &B_1,  &\ldots, &B_{k-1}, &B_{k+1}, &\ldots, &B_n\\
&C_1, &\ldots, &C_{k-1}, &C_{k+1}, &\ldots, &C_n
\end{array}\vline~\frac{x_1}{1-x_k},\ldots,\frac{x_{k-1}}{1-x_k},\frac{x_{k+1}}{1-x_k},\ldots\frac{x_n}{1-x_k}\right]\\
&\hspace{1.5cm}-\frac{\epsilon(x_k)}{q}C_k(-1)F_A^{(n-1)}
\left[\begin{array}{ccccccc}
A; &B_1,  &\ldots, &B_{k-1}, &B_{k+1}, &\ldots, &B_n\\
&C_1, &\ldots, &C_{k-1}, &C_{k+1}, &\ldots, &C_n
\end{array}\vline~x_1,\ldots,x_{k-1},x_{k+1},\ldots, x_n\right].
\end{align*}}
\end{theorem}
\begin{proof}
We first consider
\begin{align*}
M_k(t_1,\ldots,t_n)=\frac{C_k(-1)}{q}\left(\prod_{\substack{i=1\\i\neq k}}^{n}\frac{B_iC_i(-1)}{q}B_i(t_i)\overline{B}_iC_i(1-t_i)
\right)C_k(1-t_k)\overline{A}(1-x_1t_1-x_2t_2-\cdots-x_nt_n).
\end{align*}
Replacing $t_k$ by $\frac{x_1t_1+\cdots+x_{k-1}t_{k-1}+x_{k+1}t_{k+1}+\cdots+x_{n}t_{n}}{1-x_k}$, we obtain
\begin{align*}
\sum_{t_1,\ldots, t_n\in\mathbb{F}_q} M_k(t_1,\cdots,t_n)&=\frac{C_k(-1)}{q}\sum_{t_1\cdots t_n}\left(\prod_{i=1:i\neq k}^{n}\frac{B_iC_i(-1)}{q}B_i(t_i)\overline{B}_iC_i(1-t_i)
\right)\\
& \hspace{1cm} \overline{A}C_k\left(1-\frac{x_1t_1}{1-x_k}-\cdots-\frac{x_{k-1}t_{k-1}}{1-x_k}-\frac{x_{k+1}t_{k+1}}{1-x_k}\cdots-\frac{x_nt_n}{1-x_k}\right).
\end{align*}
From Definition \ref{def1}, we note that
\begin{align*}
F_A^{(n)}&\left[\begin{array}{cccccccc}
A; &B_1, &\ldots,&B_{k-1},&\epsilon, &B_{k+1},&\ldots, &B_n\\
&C_1, &\ldots,&C_{k-1},&C_k, &C_{k+1},&\ldots, &C_n
\end{array}\vline~x_1,\ldots,x_n\right]\\
&\hspace{2cm}=\epsilon(x_1,\ldots, x_n)\sum_{t_1,\ldots, t_n\in\mathbb{F}_q} M_k(t_1,\ldots,t_n)\epsilon(t_k)\notag\\
&\hspace{2cm}=\epsilon(x_1\cdots x_n)\sum_{\substack{t_1,\ldots, t_n\in\mathbb{F}_q\\t_k\neq 0}} M_k(t_1,\ldots,t_n)\notag\\
&\hspace{2cm}=\epsilon(x_1\cdots x_n)\sum_{t_1,\ldots, t_n\in\mathbb{F}_q} M_k(t_1,\ldots,t_n)-\epsilon(x_1\cdots x_n)\sum_{\substack{t_1,\ldots, t_n\in\mathbb{F}_q\\t_k=0}} M_k(t_1,\ldots,t_n).\notag
\end{align*}
Hence the result follows.
\end{proof}
\begin{theorem}\label{equal}
For $n\geq2$, let $A,B_1,\ldots,B_n,C_1,\ldots,C_n\in\widehat{\mathbb{F}_q}$ and $x_1,\ldots,x_n\in\mathbb F_q$. If $x_k\neq1$, then
{\tiny
\begin{align*}
F_A^{(n)}&\left[\begin{array}{cccccccc}
A; &B_1, &\ldots,&B_{k-1},&B_k, &B_{k+1},&\ldots, &B_n\\
&C_1, &\ldots,&C_{k-1},&B_k, &C_{k+1},&\ldots, &C_n
\end{array}\vline~x_1,\ldots,x_n\right]\\
&=\frac{\epsilon(x_k)}{q}\overline{A}(1-x_k)F_A^{(n-1)}\left[\begin{array}{ccccccc}
\overline{B_k}; &B_1, &\ldots,&B_{k-1}, &B_{k+1},  &\ldots, &B_n\\
&C_1, &\ldots, &C_{k-1}, &C_{k+1} &\ldots, &C_n
\end{array}\vline~\frac{x_1}{x_k},\ldots,\frac{x_{k-1}}{x_k},\frac{x_{k+1}}{x_k}\ldots,\frac{x_n}{x_k}\right]\\
&\hspace{.2cm}-\frac{\epsilon(x_k)}{q}\overline{A}(1-x_k)F_A^{(n-1)}\left[\begin{array}{ccccccc}
A; &B_1, &\ldots, &B_{k-1}, &B_{k+1},  &\ldots, &B_n\\
&C_1, &\ldots, &C_{k-1}, &C_{k+1}, &\ldots, &C_n
\end{array}\vline~\frac{x_1}{1-x_k},\ldots,\frac{x_{k-1}}{1-x_k},\frac{x_{k+1}}{1-x_k}\ldots,\frac{x_n}{1-x_k}\right].
\end{align*}}
\end{theorem}
\begin{proof}
We consider
\begin{align*}
N_k(t_1,\ldots,t_n)
&=\frac{1}{q}\left(\prod_{\substack{i=1\\i\neq k}}^{n}\frac{B_iC_i(-1)}{q}B_i(t_i)\overline{B}_iC_i(1-t_i)
\right)B_k(t_k)\overline{A}(1-x_1t_1-x_2t_2-\cdots-x_nt_n).
\end{align*}
Replacing $t_k$ by $1-\frac{x_1t_1+\cdots+x_{k-1}t_{k-1}+x_{k+1}t_{k+1}+\cdots+x_{n}t_{n}}{x_k}$, we have
\begin{align*}
\sum_{t_1,\ldots, t_n\in\mathbb{F}_q}N_k(t_1,\ldots,t_n)&=\frac{1}{q}\sum_{t_1\cdots t_n}\left(\prod_{\substack{i=1\\i\neq k}}^{n}\frac{B_iC_i(-1)}{q}B_i(t_i)\overline{B}_iC_i(1-t_i)
\right)\\
&\hspace{.5cm}\overline{A}(1-x_k)B_k\left(1-\frac{x_1t_1}{x_k}-\cdots-\frac{x_{k-1}t_{k-1}}{x_k}-\frac{x_{k+1}t_{k+1}}{x_k}\cdots-\frac{x_nt_n}{x_k}\right).
\end{align*}
Again,
\begin{align*}
\sum_{\substack{t_1,\ldots, t_n\in\mathbb{F}_q\\t_k=1}} N_k(t_1,\ldots,t_n)&=\frac{1}{q}\sum_{t_1,\ldots, t_n\in\mathbb{F}_q}\left(\prod_{\substack{i=1\\i\neq k}}^{n}\frac{B_iC_i(-1)}{q}B_i(t_i)\overline{B}_iC_i(1-t_i)
\right)B_k(1)\\
&\hspace{1cm}\overline{A}(1-x_1t_1-\cdots-x_{k-1}t_{k-1}-x_k-x_{k+1}t_{k+1}-\cdots-x_nt_n)\\
&=\sum_{t_1,\ldots, t_n\in\mathbb{F}_q}\left(\prod_{\substack{i=1\\i\neq k}}^{n}\frac{B_iC_i(-1)}{q}B_i(t_i)\overline{B}_iC_i(1-t_i)
\right)\\
&\hspace{.5cm}\overline{A}(1-x_k)\overline{A}\left(1-\frac{x_1t_1}{1-x_k}-\cdots-\frac{x_{k-1}t_{k-1}}{1-x_k}-\frac{x_{k+1}t_{k+1}}{1-x_k}-\cdots-\frac{x_nt_n}{1-x_k}\right).
\end{align*}
By Definition \ref{def1}, we have
\begin{align*}
&F_A^{(n)}\left[\begin{array}{cccccccc}
A; &B_1, &\ldots,&B_{k-1},&B_k, &B_{k+1},&\ldots, &B_n\\
&C_1, &\ldots,&C_{k-1},&B_k, &C_{k+1},&\ldots, &C_n
\end{array}\vline~x_1,\ldots,x_n\right]\\
&\hspace{2cm}=\epsilon(x_1\cdots x_n)\sum_{t_1,\ldots, t_n\in\mathbb{F}_q} N_k(t_1,\ldots,t_n)\epsilon(1-t_k)\notag\\
&\hspace{2cm}=\epsilon(x_1\cdots x_n)\sum_{\substack{t_1\cdots t_n\\t_k\neq1}} N_k(t_1,\cdots,t_n)\notag\\
&\hspace{2cm}=\epsilon(x_1\cdots x_n)\sum_{t_1,\ldots, t_n\in\mathbb{F}_q} N_k(t_1,\ldots,t_n)-\epsilon(x_1\cdots x_n)\sum_{\substack{t_1,\ldots, t_n\in\mathbb{F}_q\\t_k=1}} N_k(t_1,\ldots,t_n).\notag
\end{align*}
Thus we complete the proof of the theorem.
\end{proof}
\section{Generating Functions}
In this section, we deduce finite field analogue of certain generating functions for the Lauricella series $F_A^{(n)}$. 
\begin{theorem}\label{gen}
For any $a\in\mathbb C$ and $\mid t\mid <1$, we have
	\begin{align*}
	&\sum_{k=0}^{\infty}{a+k-1\choose k}F_A^{(n)}\left[\begin{array}{cccc}
	a+k; &b_1, &\ldots, &b_n\\
	&c_1, &\ldots, &c_n
	\end{array}\vline~x_1,\ldots,x_n\right]t^k\\
	&\hspace{4cm}=(1-t)^{-a}F_A^{(n)}\left[\begin{array}{cccc}
	a; &b_1, &\ldots, &b_n\\
	&c_1, &\ldots, &c_n
	\end{array}\vline~\frac{x_1}{1-t},\ldots,\frac{x_n}{1-t}\right].
	\end{align*}
\end{theorem}
\begin{proof}
Using
$${a\choose n}=\frac{\Gamma(a+1)}{n!~\Gamma(a-n+1)}~\hbox{and}~
(a)_{m}=\frac{\Gamma(a+m)}{\Gamma(a)},$$
we have
\begin{align*}
\sum_{k=0}^{\infty}&{a+k-1\choose k}F_A^{(n)}\left[\begin{array}{cccc}
a+k; &b_1, &\ldots, &b_n\\
&c_1, &\ldots, &c_n
\end{array}\vline~x_1,\ldots,x_n\right]t^k\\
&=\sum_{k=0}^{\infty}{a+k-1\choose k}
\left[\sum_{\substack{m_i\geq0\\1\leq i\leq n}}\frac{(a+k)_{m_1+\cdots+m_n}(b_1)_{m_1}\cdots(b_n)_{m_n}}{(c_1)_{m_1}\cdots(c_n)_{m_n}}\frac{x_1^{m_1}}{m_1!}\cdots\frac{x_n^{m_n}}{m_n!}\right]t^k\\
&=\sum_{\substack{k,m_i\geq0\\1\leq i\leq n}}\left[\frac{\Gamma(a+m_1+\cdots+m_n+k)}{k!~\Gamma(a)}\frac{(b_1)_{m_1}\cdots(b_n)_{m_n}}{(c_1)_{m_1}\cdots(c_n)_{m_n}}\frac{x_1^{m_1}}{m_1!}\cdots\frac{x_n^{m_n}}{m_n!}\right]t^k\\
&=\sum_{\substack{m_i\geq0\\1\leq i\leq n}}\left[\frac{\Gamma(a+m_1+\cdots+m_n)}{\Gamma(a)}\frac{(b_1)_{m_1}\cdots(b_n)_{m_n}}{(c_1)_{m_1}\cdots(c_n)_{m_n}}\frac{x_1^{m_1}}{m_1!}\cdots\frac{x_n^{m_n}}{m_n!}\left[\sum_{k=0}^{\infty}\frac{\Gamma(a+m_1+\cdots+m_n+k)}{k!~\Gamma(\Gamma(a+m_1+\cdots+m_n))}\right]\right]t^k\\
&=\sum_{\substack{m_i\geq0\\1\leq i\leq n}}\left[\frac{(a)_{m_1+\cdots+m_n}(b_1)_{m_1}\cdots(b_n)_{m_n}}{(c_1)_{m_1}\cdots(c_n)_{m_n}}\frac{x_1^{m_1}}{m_1!}\cdots\frac{x_n^{m_n}}{m_n!}\left[\sum_{k=0}^{\infty}{a+m_1+\cdots+m_n+k-1\choose k}t^k\right]\right].
\end{align*}
Using the fact that
$$\sum_{k=0}^{\infty}{a+m_1+\cdots+m_n+k-1\choose k}t^k=(1-t)^{-(a+m_1+\cdots+m_n)},~\mid t<1,$$
we complete the proof of the theorem.
\end{proof}
The following theorem gives a finite field analogue of Theorem \ref{gen}.
\begin{theorem}\label{gen1}
Let $A,B_1,\ldots,B_n,C_1,\ldots,C_n\in\widehat{\mathbb{F}_q}$, $x_1,\ldots,x_n\in\mathbb F_q$ and $t\in\mathbb F_q^\times\backslash\lbrace1\rbrace$. Then
\begin{align*}
\frac{q}{q-1}&\sum_{\theta\in\widehat{\mathbb F_q}}{A\theta\choose\theta}F_A^{(n)}\left[\begin{array}{cccc}
A\theta; &B_1, &\ldots, &B_n\\
&C_1, &\ldots, &C_n
\end{array}\vline~x_1,\ldots,x_n\right]\theta(t)\\
&\hspace{2cm}=\overline{A}(1-t)F_A^{(n)}\left[\begin{array}{cccc}
A; &B_1, &\ldots, &B_n\\
&C_1, &\ldots, &C_n
\end{array}\vline~\frac{x_1}{1-t},\ldots,\frac{x_n}{1-t}\right].
\end{align*}
\end{theorem}
\begin{proof} Using \eqref{greenevar} and noting that $x_1t_1+\cdots+x_nt_n\neq1$, we have
\begin{align*}
\frac{q}{q-1}&\sum_{\theta\in\widehat{\mathbb F_q}}{A\theta\choose\theta}F_A^{(n)}\left[\begin{array}{cccc}
A\theta; &B_1, &\ldots, &B_n\\
&C_1, &\ldots, &C_n
\end{array}\vline~x_1,\ldots,x_n\right]\theta(t)\\
&=\frac{q}{q-1}\sum_{\theta\in\widehat{\mathbb F_q}}{A\theta\choose\theta}\sum_{t_1,\ldots, t_n\in\mathbb{F}_q}\left(\prod_{i=1}^{n}\epsilon(x_i)\frac{B_iC_i(-1)}{q}B_i(t_i)\overline{B}_iC_i(1-t_i)
\right)\\
&\hspace{6cm}\overline{A\theta}(1-x_1t_1-\cdots -x_nt_n)\theta(t)\\
&=\frac{q}{q-1}\sum_{t_1,\ldots, t_n\in\mathbb{F}_q}\left(\prod_{i=1}^{n}\epsilon(x_i)\frac{B_iC_i(-1)}{q}B_i(t_i)\overline{B}_iC_i(1-t_i)
\right)\\
&\hspace{2cm}\overline{A}(1-x_1t_1-\cdots -x_nt_n)\sum_{\theta\in\widehat{\mathbb F_q}}{A\theta\choose\theta}\theta\left(\frac{t}{1-x_1t_1-\cdots-x_nt_n}\right)\\
&=\sum_{t_1,\ldots, t_n\in\mathbb{F}_q}\left(\prod_{i=1}^{n}\epsilon(x_i)\frac{B_iC_i(-1)}{q}B_i(t_i)\overline{B}_iC_i(1-t_i)
\right)\overline{A}(1-t)\overline{A}\left(1-\frac{x_1t_1}{1-t}-\cdots-\frac{x_nt_n}{1-t}\right).
\end{align*}
Thus the result follows from Definition \ref{def1}.
\end{proof}
In the following theorem, we establish another generating function for $F_A^{(n)}$.
\begin{theorem}\label{gen2}
	Let $A,B_1,\ldots,B_n,C_1,\ldots,C_n\in\widehat{\mathbb{F}_q}$, $x_1,\ldots,x_n\in\mathbb F_q$, and $t\in\mathbb F_q^\times\backslash\lbrace1\rbrace$. Then
	\begin{align*}
	\frac{q}{q-1}&\sum_{\theta\in\widehat{\mathbb F_q}}{A\theta\choose\theta}F_A^{(n)}\left[\begin{array}{cccc}
	\overline{\theta}; &B_1, &\ldots, &B_n\\
	&C_1, &\ldots, &C_n
	\end{array}\vline~x_1,\ldots,x_n\right]\theta(t)\\
	&=\overline{A}(1-t)F_A^{(n)}\left[\begin{array}{cccc}
	A; &B_1, &\ldots, &B_n\\
	&C_1, &\ldots, &C_n
	\end{array}\vline~-\frac{x_1t}{1-t},\ldots,-\frac{x_nt}{1-t}\right].
	\end{align*}
\end{theorem}
\begin{proof}
Following steps similar to the proof of Theorem \ref{gen1}, the result can be easily obtained.
\end{proof}
For $n=2$, we have a generating function for $F_2\left[\begin{array}{cccc}
A; &B_1, &B_2\\
&C_1, &C_2
\end{array}\vline~x_1,x_2\right]$, which appears to be new.
\begin{corollary}
Let $A,B,B',C,C'\in\widehat{\mathbb{F}_q}$, $x_1,x_2\in\mathbb F_q$, and $t\in\mathbb F_q^\times\backslash\lbrace1\rbrace$, then
\begin{align*}
\frac{q}{q-1}\sum_{\theta}{A\theta\choose\theta}&F_A^{(2)}\left[\begin{array}{ccc}
\overline{\theta}; &B, &B'\\
&C, &C'
\end{array}\vline~x_1,x_2\right]\theta(t)\\
&=\overline{A}(1-t)F_A^{(2)}\left[\begin{array}{cccc}
A; &B, &B'\\
&C, &C'
\end{array}\vline~-\frac{x_1t}{1-t},-\frac{x_2t}{1-t}\right].
\end{align*}
\end{corollary}
We finally deduce another generating function for $F^{(n)}_A$ over finite field in the following theorem.
\begin{theorem}\label{gen3}
Let $A,B_1,\ldots,B_n,C_1,\ldots,C_n\in\widehat{\mathbb{F}_q}$, $x_1,\ldots,x_n\in\mathbb F_q$ and $t\in\mathbb F_q^\times\backslash\lbrace1\rbrace$. For $1\leq k\leq n$, we have
\begin{align*}
\frac{q}{q-1}&\sum_{\theta\in\widehat{\mathbb F_q}}{B_k\overline{C_k}\theta\choose\theta}F_A^{(n)}\left[\begin{array}{cccccc}
A; &B_1, &\ldots, &B_k\theta, &\ldots &B_n\\
&C_1, &\ldots, &C_k, &\ldots, &C_n
\end{array}\vline~x_1,\ldots,x_n\right]\theta(t)\\
&\hspace{3cm}=\overline{B_k}(1-t)F_A^{(n)}\left[\begin{array}{cccc}
A; &B_1, &\ldots, &B_n\\
&C_1, &\ldots, &C_n
\end{array}\vline~x_1,\ldots\frac{x_k}{1-t},\ldots,x_n\right].
\end{align*}
\end{theorem}
\begin{proof}
Using \eqref{greenevar}, we have
\begin{align*}
\frac{q}{q-1}&\sum_{\theta\in\widehat{\mathbb F_q}}{B_k\overline{C_k}\theta\choose\theta}F_A^{(n)}\left[\begin{array}{cccccc}
A; &B_1, &\cdots, &B_k\theta, &\cdots &B_n\\
&C_1, &\cdots, &C_k, &\cdots, &C_n
\end{array}\vline~x_1,\cdots,x_n\right]\theta(t)\\
&=\frac{q}{q-1}\sum_{\theta\in\widehat{\mathbb F_q}}{B_k\overline{C_k}\theta\choose\theta}\sum_{t_1,\ldots, t_n\in\mathbb{F}_q}\left(\prod_{i=1}^{n}\epsilon(x_i)\frac{B_iC_i(-1)}{q}B_i(t_i)\overline{B}_iC_i(1-t_i)
\right)\\
&\hspace{6cm}\overline{A}(1-x_1t_1-\cdots -x_nt_n)\theta(-t_k)\overline{\theta}(1-t_k)\\
&=\frac{q}{q-1}\sum_{t_1,\ldots, t_n\in\mathbb{F}_q}\left(\prod_{i=1}^{n}\epsilon(x_i)\frac{B_iC_i(-1)}{q}B_i(t_i)\overline{B}_iC_i(1-t_i)
\right)\\
&\hspace{4cm}\overline{A}(1-x_1t_1-\cdots -x_nt_n)\sum_{\theta\in\widehat{\mathbb F_q}}{B_k\overline{C_k}\theta\choose\theta}\theta\left(-\frac{tt_k}{1-t_k}\right)\\
&=\frac{q}{q-1}\sum_{t_1,\ldots, t_n\in\mathbb{F}_q}\left(\prod_{\substack{i=1\\i\neq k}}^{n}\epsilon(x_i)\frac{B_iC_i(-1)}{q}B_i(t_i)\overline{B}_iC_i(1-t_i)
\right)B_k(t_k)\overline{B}_kC_k(1-t_k)\\
&\hspace{4cm}\overline{A}(1-x_1t_1-\cdots -x_nt_n)\overline{B}_kC_k\left(1+\frac{tt_k}{1-t_k}\right)\\
&=\frac{q}{q-1}\sum_{t_1,\ldots, t_n\in\mathbb{F}_q}\left(\prod_{\substack{i=1\\i\neq k}}^{n}\epsilon(x_i)\frac{B_iC_i(-1)}{q}B_i(t_i)\overline{B}_iC_i(1-t_i)
\right)B_k(t_k)\\
&\hspace{4cm}\overline{A}(1-x_1t_1-\cdots -x_nt_n)\overline{B}_kC_k\left(1-t_k(1-t)\right).
\end{align*}
Replacing $t_k$ by $\frac{t_k'}{1-t}$, we obtain the desired result.
\end{proof}

\end{document}